\documentclass[12pt, reqno]{amsart}
\usepackage{amssymb,dsfont}
\usepackage{eucal}
\usepackage{amsmath}
\usepackage{amscd}
\usepackage[dvips]{color}
\usepackage{multicol}
\usepackage[all]{xy}           
\usepackage{graphicx}
\usepackage{color}
\usepackage{colordvi}
\usepackage{xspace}
\usepackage{txfonts}
\usepackage{lscape}
\usepackage{tikz} 

\numberwithin{equation}{section}
\usepackage[shortlabels]{enumitem}
\usepackage{ifpdf}
\ifpdf
 \usepackage[colorlinks,final,backref=page,hyperindex]{hyperref}
\else
\fi
\usepackage{tikz}
\usepackage[active]{srcltx}

\makeatletter

\newtheorem{theorem}{Theorem}[section]
\newtheorem{definition}[theorem]{Definition}
\newtheorem{lemma}[theorem]{Lemma}
\newtheorem{proposition}[theorem]{Proposition}
\newtheorem{remark}[theorem]{Remark}
\newtheorem{claim}{Claim}
\numberwithin{equation}{section}

\def\bfi{\mathbf{i}}
\def\bfj{\mathbf{j}}

\def\al{\alpha}
\def\be{\beta}

\def\h{\mathfrak{h}}
\def\mg{\mathfrak{g}}
\def \<{\langle}
\def \>{\rangle}

\def\GG{\mathcal{G}}

\def\UU{\mathcal{U}}
\def\VV{\mathcal{V}}

\def\HH{\mathcal{H}}
\def\DD{\mathcal{D}}
\def\RR{\mathcal{R}}

\newcommand{\C}{\mathbb C}

\newcommand{\N}{\mathbb{N}}
\newcommand{\Z}{\mathbb{Z}}

\def\Ind{\mathrm{Ind}}
\def\supp{\mathrm{supp}}

\begin{document}
\title[Non-weight modules]{Non-weight modules over the mirror  Heisenberg-Virasoro algebra}
 \author{Dongfang Gao}
  \address{D. Gao: School of Mathematical Sciences, University of Science and Technology of China,
  Hefei, Anhui 230026, P. R. China}
  \email{gaodfw@mail.ustc.edu.cn}
  \author{Yao Ma}
  \address{Y. Ma: School of Mathematics and Statistics, Northeast Normal University, Changchun, Jilin 130024,   P. R. China}
  \email{may703@nenu.edu.cn}
  \author{Kaiming Zhao}
  \address{K. Zhao: Department of Mathematics, Wilfrid Laurier University,
  Waterloo, ON N2L 3C5, Canada,  and School of
		Mathematical Science, Hebei Normal (Teachers)
		University, Shijiazhuang, Hebei, 050024 P. R. China.}
  \email{kzhao@wlu.ca}
\date{}
\maketitle

\begin{abstract}
In this paper, we study irreducible non-weight modules over the mirror  Heisenberg-Virasoro algebra $\DD$, including Whittaker modules, $\UU(\C d_0)$-free modules,  and their tensor products. More precisely, we give the necessary and sufficient conditions for the Whittaker modules to be irreducible. We determine all $\DD$-module structures on $\UU(\C d_0)$,  and find  the    necessary and sufficient conditions for   these modules  to be irreducible. At last we determine the necessary and sufficient conditions for the tensor products of Whittaker modules and $\UU(\C d_0)$-free modules to be irreducible, and    obtain that any two such tensor products are isomorphic if and only if the corresponding Whittaker modules and $\UU(\C d_0)$-free modules are isomorphic. These lead to many new irreducible non-weight modules over $\DD$.
\end{abstract}

\section{Introduction}
It is well-known that the Virasoro algebra $\VV$ is one of the most important Lie algebras
in mathematics and in mathematical physics  because of its widespread applications
  in  quantum physics (see \cite{GO}), conformal field theory (see \cite{DMS}),
vertex operator algebras (see \cite{DMZ,FZ}), and so on.
Many other interesting and important algebras are closely related to the Virasoro algebra, such as the  Schr{\"o}dinger-Virasoro algebra (see \cite{H,H1}), the  twisted  Heisenberg-Virasoro algebra (see \cite{ACKP, JJ, LZ1}) and the mirror  Heisenberg-Virasoro algebra $\DD$ (see \cite{B,GZ,LPXZ}) which is the even part of the mirror  $N=2$ superconformal algebra (see \cite{B}).
The mirror  Heisenberg-Virasoro algebra $\DD$  has a nice structure (see Definition \ref{def.1})
which is similar to the  twisted Heisenberg-Virasoro algebra.
This algebra is the main object   we concern in this paper.

Harish-Chandra modules and weight modules with infinite-dimensional weight spaces
have been the most popular modules in the   representation theory for many Lie algebras with triangular decomposition $\GG=\GG_+\oplus \h\oplus \GG_-$.
To some extent, Harish-Chandra modules are well understood for many infinite-dimensional Lie algebras, for example,
the affine Kac-Moody algebras in \cite{CP, GZ1}, the Virasoro algebra in \cite{A,FF,M},
the twisted Heisenberg-Virasoro algebra  in \cite{ACKP, LZ1},
the Schr{\"o}dinger-Virasoro algebra in \cite{LS,TZ1},
and the mirror  Heisenberg-Virasoro algebra in \cite{LPXZ}.
There are also some researches about weight modules with infinite-dimensional weight spaces
(see \cite{BBFK, CGZ,GZ,LZ3}).

Recently, non-weight modules over $\GG$ attract more attentions from mathematicians. In particular,  Whittaker modules, and   $\UU(\h)$-free $\GG$-modules    have been widely studied for many Lie algebras.
Whittaker modules  for $\mathrm{sl_2(\C)}$ were  constructed by Arnal and Pinzcon (see \cite{AP}).
Whittaker modules  for arbitrary finite-dimensional complex semisimple Lie algebra $\mg$ were  introduced and systematically studied by Kostant in \cite{Ko},
where he proved that these modules
with a fixed regular Whittaker function (Lie homomorphism) on a nilpotent
radical are (up to isomorphism) in bijective correspondence with
central characters of $\UU(\GG)$.
In recent years,   Whittaker modules for many other Lie algebras have been investigated
(see \cite{AHPY,ALZ,BM,BO,C,MD1,MD2}).
The notation of $\UU(\h)$-free modules was first introduced by Nilsson \cite{N} for the simple Lie algebra $\mathrm{sl}_{n+1}(\C)$. At the same time
these modules were introduced in a very different approach in the paper \cite{TZ}.
Later, $\UU(\h)$-free modules for many infinite-dimensional Lie algebras are determined, for example,
the Kac-Moody algebras in \cite{CTZ}, the Virasoro algebra in \cite{LZ2,TZ},
the Witt algebra in \cite{TZ}, the twisted Heisenberg-Virasoro algebra and $W(2,2)$ algebra in \cite{CG}, and so on.
In the present paper, we will study the Whittaker modules and $\UU(\C d_0)$-free modules over $\DD$.
Also, we study the tensor products of Whittaker modules and $\UU(\C d_0)$-free modules.

The paper is organized as follows. In Section 2, we recall  notations related to the
mirror  Heisenberg-Virasoro algebra $\DD$ and
collect some known results on Whittaker modules and $\UU(\C d_0)$-free modules over $\VV$,
including also two important lemmas on tensor product modules for later use.
In Section 3, we give the necessary and sufficient conditions for the Whittaker modules $W_{\varphi_m}$ over $\DD$
to be irreducible, see Theorem \ref{main3.5}.
In Section 4, we determine all $\DD$-module structures on $\UU(\C d_0)$,  and find  the    necessary and sufficient conditions for   these modules  to be irreducible,
see Theorem \ref{main4.2}.
In Section 5, we give the necessary and sufficient conditions for the tensor product of a Whittaker module $W_{\varphi_m}$
and a $\UU(\C d_0)$-free  $\DD$-module $\Omega(\lambda,\al,\be)$ to be irreducible, see Theorem \ref{main5.3}.
Furthermore, we show that two such tensor product modules are isomorphic if and only if the corresponding
Whittaker modules and $\UU(\C d_0)$-free  $\DD$-modules are isomorphic, see Theorem \ref{main5.4}. Consequently, we obtain a lot of new irreducible non-weight modules over $\DD$.

Throughout this paper, we denote by $\Z,\N,\Z_+, \C$ and $\C^*$
the set of integers, positive integers, non-negative integers, complex numbers and nonzero complex numbers respectively. All vector spaces and Lie algebras are over $\C$. We denote by $\UU{(\GG)}$ the universal enveloping algebra for a Lie algebra $\GG$.

\section{Notations and Preliminaries}
 For     convenience, in this section we recall some notations and collect some known results.

\begin{definition}\label{def.1}
The {\bf mirror  Heisenberg-Virasoro algebra} $\DD$ is a Lie algebra with basis\\
$\{d_m, h_r,{\bf{c}},{\bf{l}}~|~m\in\Z, r\in \frac{1}{2}+\Z\}$ subject to  the following commutation relations
\begin{equation*}
\begin{split}
&[d_m, d_n]=(m-n)d_{m+n}+\frac{m^3-m}{12}{\delta}_{m+n, 0}{{\bf{c}}},\\
&[d_m, h_r]=-rh_{m+r},                 \\
&[h_r, h_s]=r{\delta}_{r+s, 0}{\bf{l}},\\
&[{\bf{c}},\DD]=[{\bf{l}},\DD]=0,
\end{split}
\end{equation*}
for $m,n\in\Z, r,s\in \frac{1}{2}+\Z$.
\end{definition}
The Lie subalgebra spanned by $\{d_m, {\bf{c}}~|~ m\in\Z \}$ is the Virasoro algebra $\VV$,
and the Lie subalgebra spanned by $\{h_r, {\bf{l}}~|~r\in \frac{1}{2}+\Z \}$
is the twisted Heisenberg algebra $\HH$.
Moreover, $\HH$ is an ideal of $\DD$
and $\DD$ is the semi-direct product of $\VV$ and $\HH$.
It is clear that $\DD, \VV, \HH$ are all $\Z$-graded Lie algebras.


\begin{definition}
Let $\GG=\oplus_{i\in\Z}{\GG}_{i}$ (resp. $\GG=\oplus_{i\in\frac12\Z}{\GG}_{i}$) be a $\Z$ (resp.  $\frac12\Z$)-graded Lie algebra.
A $\GG$-module $V$ is called  the $\bf{restricted}$ module if for any $v\in V$ there exists $n\in\N$ such that $\GG_iv=0$,
for $i>n$ (resp. $i>\frac12n$). The  category of restricted modules over $\GG$ will be denoted as $\RR_{\GG}$.
\end{definition}
Note that if $V$ is a $\VV$-module,
then $V$ can be easily viewed as a $\DD$-module by defining $\HH V=0$,
the resulting  module is denoted by $V^{\DD}$.
Thanks to \cite{FLM}, for any $H\in\RR_{\HH}$ with the action of $\bf l$ as a nonzero scalar $l$, we can give $H$ a $\DD$-module structure
denoted by $H^{{\DD}}$ via the following map
\begin{align}
d_n&\mapsto \frac{1}{2l}\sum_{k\in\Z+\frac{1}{2}}h_{n-k}h_k,\quad\forall n\in\Z, n\neq0,\label{rep1}\\
d_0&\mapsto \frac{1}{2l}\sum_{k\in\Z+\frac{1}{2}}h_{-|k|}h_{|k|}+\frac{1}{16},\label{rep2}\\
h_r&\mapsto h_r,\quad\forall r\in\frac{1}{2}+\Z,
\quad {\bf{c}} \mapsto 1,
\quad {\bf{l}}\mapsto l.\label{rep3}
\end{align}

\begin{definition}
Let $\GG$ be a Lie algebra. Suppose $f$ is an automorphism of $\GG$ and $V$ is a $\GG$-module. 
The following actions
$$x\cdot v=f(x)v,\quad \forall x\in\GG,v\in V$$
give $V$ a new $\GG$-module structure, denoted by $V'.$ Then $V,V'$ are called {\bf equivalent} $\GG$-modules.
\end{definition}
\begin{remark}
It is easy to see that equivalent modules over the Lie algebra have the same irreducibility.
\end{remark}

For convenience, we define the following subalgebras. For any $m,n\in\Z_+$, set
\begin{equation*}
\begin{split}
&\DD^{(m, -n)}=\sum_{i\ge m}\C d_i\oplus \sum_{i\in\Z_+}\C h_{-n+i+\frac12}\oplus\C{\bf c}\oplus \C{\bf l}, \\
&\DD^{(m,-\infty)}=\sum_{i\in\Z_+}\C d_{m+i}\oplus \sum_{i\in\Z}\C h_{i+\frac12}+\C{\bf c}+\C{\bf l},\\
&\VV^{(m)}= \sum_{i\in\Z_+}\C d_{m+i}+\C{\bf c},\\
&\HH^{(m)}=\sum_{i\in\Z_+}\C h_{m+i+\frac12}+\C{\bf l}.
\end{split}
\end{equation*}

\begin{lemma}(cf, \cite[Theorem 6.2]{LPXZ})\label{main2.3}
	Let $V$ be an irreducible  ${\DD}^{(0, -q)}$-module for some $q\in \Z_+$ such that the action of {\bf c, l} on $V$ are the scalars  $c$ and $0$ respectively. Assume that there exists an integer $t\ge-q$ satisfying the following two conditions:
	\begin{enumerate}[$(a)$]
		\item the action of  $h_{t+\frac12}$ on $V$ is bijective;
		\item $h_{n+\frac{1}{2}}V=0=d_{n+q}V$ for all $n> t$.
	\end{enumerate}
	Then the induced ${\DD}$-module $\mathrm{Ind}_{\DD^{(0, -q)}}^{\DD}(V)$ is  irreducible.
\end{lemma}

Recall that for any $m\in\Z_+$, let $\psi_m: \VV^{(m)}\rightarrow \C$ be a Whittaker function, i.e., a Lie algebra homomorphism.
Then we have the one-dimensional module $\C w_{\psi_m}$ over $\VV^{(m)}$
with $x\cdot w_{\psi_m}=\psi_m(x)w_{\psi_m}$, for $x\in\VV^{(m)}$.
The induced $\VV$-module
\begin{equation}\label{whittaker-1}
W_{\psi_m}=\Ind_{\VV^{(m)}}^{\VV} \C w_{\psi_m}
\end{equation}
is called the {\bf  Whittaker module} over $\VV$ with respect to $\psi_m$.

\begin{lemma}(cf, \cite[Theorem 7]{LGZ}) \label{main2.5}
 For any $m\in\N$, and any Whittaker function $\psi_m: \VV^{(m)}\rightarrow \C$, then the Whittaker module $W_{\psi_m}$ is irreducible if and only if $(\psi(d_{2m}),\psi(d_{2m-1}))\ne (0,0)$.
\end{lemma}

 For $\lambda\in\C^*,\al\in\C$, denote by $\Omega(\lambda, \al)=\C[t]$ the polynomial algebra over $\C$.
It is well-known that $\Omega(\lambda, \al)$ is a $\VV$-module with the actions
$${\bf c}(f(t))=0,\ \ d_m(f(t))=\lambda^m(t+m\al)f(t+m), ~\forall~ m\in\Z.$$
 Thanks to \cite{LZ2}, we know that $\Omega(\lambda, \al)$ is irreducible if and only if $\al\ne 0$,
and $\Omega(\lambda, 0)$ has an irreducible submodule $t\Omega(\lambda, \al)$.
Moreover, we have the following lemma.
\begin{lemma}(cf, \cite[Theorem 3]{TZ}) \label{main2.6}
Let $V$ be a $ \VV$-module.
Assume that $V$ viewed as a $\UU(\C d_0)$-module is free of rank $1$. Then $V\cong \Omega(\lambda, \al)$ for some $\lambda\in\C^*,\al\in\C$.
\end{lemma}

We conclude this section by recalling two results about tensor product for later use.
\begin{lemma}(cf, \cite[Lemma 8]{LZ3})\label{main2.7}
Let $V$ be a module over a Lie algebra $\GG$, $\mathfrak{g}$ be a subalgebra of $\GG$, and $W$ be a $\mathfrak{g}$-module.
Then the $\GG$-module homomorphism $\tau:
\Ind_{\mathfrak{g}}^{\GG}(V\otimes W)\rightarrow V\otimes
\Ind_{\mathfrak{g}}^{\GG}(W)$ induced from the inclusion map $V\otimes
W\rightarrow V\otimes \Ind_{\mathfrak{g}}^{\GG}(W)$ is a $\GG$-module isomorphism.
\end{lemma}

\begin{lemma}(cf, \cite[Lemma 3.1]{GZ})\label{main2.8}
Suppose $V, W$ are  $\VV$-modules, and $H,K\in \RR_{\HH}$ are irreducible with nonzero action of $\bf{l}$.
Then
 \begin{enumerate}[$(1)$]
 \item  any  $\DD$-submodule of  $V^{\DD}\otimes H^{\DD}$ is of the
form $(V')^{\DD}\otimes  H^{\DD}$ for some $\VV$-submodule $V'$ of
$V$. In particular, $V^{\DD}\otimes H^{\DD}$ is an irreducible $\DD$-module
if and only if $V$ is an irreducible $\VV$-module;
 \item $V^{\DD}\otimes H^{\DD}\cong W^{\DD}\otimes K^{\DD}$ if and
only if $V\cong W$ and $H\cong K$.
 \end{enumerate}
\end{lemma}

\section{Whittaker modules over $\DD$}
In this section, we will determine the necessary and sufficient conditions for the Whittaker modules over $\DD$
to be irreducible.

Assume that $\phi:\HH^{(0)}\rightarrow \C$ is a Whittaker function,
and that $\C w_{\phi}$ is the one-dimensional module over $\HH^{(0)}$
defined by  $x\cdot w_{\phi}=\phi(x)w_{\phi}$, for $x\in\HH^{(0)}$.
Then for the Whittaker  $\HH$-module
$$W_{\phi}=\Ind_{\HH^{(0)}}^{\HH} \C w_{\phi},$$
we have the following result.
\begin{lemma}\label{main3.1}
The Whittaker  $\HH$-module $W_{\phi}$ is irreducible if and only if $\phi({\bf l})\ne 0$.
\end{lemma}
\begin{proof}
We first assume that $\phi({\bf l})\ne 0$.
For any nonzero $v\in W_{\phi}$, $v$ is a linear combination of vectors in the form
$$h_{-r-\frac12}^{i_r}\cdots h_{-1-\frac12}^{i_1}h_{-\frac12}^{i_0}w_{\phi},$$
where $i_0,i_1,\cdots, i_r\in\Z_+$ by PBW Theorem.
Let $\< v\>$ denote the submodule generated by $v$ of $W_{\phi}$.
It is not hard to see $w_{\phi}\in \<v\>$. Thus $\<v\>=W_{\phi}$, which implies that $W_{\phi}$ is irreducible.

Now we assume that $\phi({\bf l})=0$. Then it is easy to see that
$h_{-\frac12}w_{\phi}$ is a Whittaker vector, which implies that $h_{-\frac12}w_{\phi}$ generates a nonzero proper submodule of $W_{\phi}$.  So $W_{\phi}$ is not irreducible.
\end{proof}

Now for any $m\in \Z_+$, let $\varphi_m:\DD^{(m, 0)}\rightarrow \C$
be a Whittaker function, and $\C w_{\varphi_m}$ be the one-dimensional
module over $\DD^{(m, 0)}$ defined by  $x\cdot w_{\varphi_m}=\varphi_m(x)w_{\varphi_m}$, for $x\in\DD^{(m, 0)}$.
In the rest of this section, we will determine the irreducibility of the Whittaker $\DD$-module
$$W_{\varphi_m}=\Ind_{\DD^{(m, 0)}}^{\DD} \C w_{\varphi_m}.$$
Each irreducible quotient of $W_{\varphi_m}$ (if it exists) is called the irreducible Whittaker $\DD$-module with respect to $\varphi_m$.
It is clear that $$\varphi_m(d_{2m+j})=\varphi_m(h_{m+j-\frac12})=0,\,\,\forall  j\in\N,$$
 since $\varphi_m([\DD^{(m, 0)},\DD^{(m, 0)}])=0$.
In particular, if $m=0$, then $W_{\varphi_m}$ is a Verma module,
which has been studied in \cite{LPXZ}.

First we consider $W_{\varphi_m}$ with $\varphi_m({\bf l})=l\ne 0$. We define a new Whittaker function
$\varphi'_m: \VV^{(m)}  \rightarrow \C$ as follows
\begin{eqnarray*}
&&\varphi'_m({\bf c})=\varphi_m({\bf c})-1,\\
&&\varphi'_m(d_{k})=0,\quad\forall k\ge 2m+1,\\
&&\varphi'_m(d_k)=\varphi_m(d_k)-\frac{1}{2l}\sum_{i=0}^{m-1}
\varphi_m(h_{i+\frac12})\varphi_m(h_{k-i-\frac12})-\delta_{0,k}\frac{1}{16},\,\, \forall m\le k\le 2m.
\end{eqnarray*}
Then we have  the Whittaker $\VV$-module
\begin{equation*}
{W}_{\varphi'_m}=\Ind_{\VV^{(m)}}^{\VV} \C w_{\varphi'_m},
\end{equation*}
where $\C w_{\varphi'_m}$ is the one-dimensional module over $\VV^{(m)}$ defined by  $x\cdot
w_{\varphi'_m}=\varphi'_m(x)w_{\varphi'_m}$, for $x\in \VV^{(m)}$.

\begin{proposition}\label{prop.1}
Suppose that $m\in \Z_+$, and $\varphi_m, \varphi'_m$ are given as above with $\varphi_m({\bf l}) \ne 0$.
Let $H=U(\HH)w_{\varphi_m}$ be in $W_{\varphi_m}$.
The following observations hold.
\begin{enumerate}[$(1)$]
\item $W_{\varphi_m}\cong H^{\DD}\otimes W_{\varphi'_m}^{\DD}$.
In particular, this agrees with \cite[Proposition 5.1]{LPXZ} if $m=0$.
\item Suppose $m\in\N$. Then $W_{\varphi_m}$ is an irreducible $\DD$-module if and only if
$$\varphi_m(d_{2m})\ne0 \,\,{\text{or}}\,\, 2\varphi_m({\bf l}) \varphi_m(d_{2m-1})-\varphi_m(h_{m-\frac12})^2\ne0.$$
\item Each irreducible Whittaker module over $\DD$ with respect to $\varphi_m$
is isomorphic to $H^{\DD}\otimes Q^{\DD}$, where $Q$ is an irreducible quotient of $\VV$-module $W_{\varphi'_m}$.
\end{enumerate}
\end{proposition}
\begin{proof}
 (1) For $\VV^{(m)}$-module $\C w_{\varphi'_m}$,
 we can easily give $\C w_{\varphi'_m}$ a $\DD^{(m,-\infty)}$-module structure by defining $\HH w_{\varphi'_m}=0$.
 The resulting module is denoted by $\C_{\varphi'_m}^{\DD^{(m,-\infty)}}$.
 By PBW Theorem,  we can define the linear map :
 $$\aligned \pi: \,\,&\Ind_{\DD^{(m, 0)}}^{\DD^{(m,-\infty)}}\C w_{\varphi_m}\rightarrow H^{\DD}\otimes \C_{\varphi'_m}^{\DD^{(m,-\infty)}},\\
 &
  h_{-r-\frac12}^{q_{r}}\cdots h_{-1-\frac12}^{q_1}h_{-\frac12}^{q_0}w_{\varphi_m}\mapsto
h_{-r-\frac12}^{q_{r}}\cdots h_{-1-\frac12}^{q_1}h_{-\frac12}^{q_0}w_{\varphi_m}\otimes w_{\varphi'_m}.\endaligned$$
It is clear that $\pi$ is a bijection.
\begin{claim}\label{cla3.2.1}
$\pi(h w_{\varphi_m})=h w_{\varphi_m}\otimes w_{\varphi'_m}$
for $h\in\UU(\HH)$.
\end{claim}
For $h\in\UU(\HH)$, $h$ can be written as a linear combination of vectors in the form
$$h_{-r_1-\frac12}^{q_{r_1}}\cdots h_{-1-\frac12}^{q_1}h_{-\frac12}^{q_0}{\bf l}^th_{\frac12}^{p_0}h_{1+\frac12}^{p_1}\cdots h_{s_1+\frac12}^{p_{s_1}}$$
by PBW Theorem.
Without loss of generality, we can assume $$h=h_{-r_1-\frac12}^{q_{r_1}}\cdots h_{-1-\frac12}^{q_1}h_{-\frac12}^{q_0}{\bf l}^th_{\frac12}^{p_0}h_{1+\frac12}^{p_1}\cdots h_{s_1+\frac12}^{p_{s_1}},$$
where $q_i,p_j,t,r_1,s_1\in\Z_+, $ for $0\leq i\leq r_1,0\leq j\leq s_1$. We compute
\begin{equation*}
\begin{split}
\pi(h w_{\varphi_m})=&\pi(h_{-r_1-\frac12}^{q_{r_1}}\cdots h_{-1-\frac12}^{q_1}h_{-\frac12}^{q_0}{\bf l}^th_{\frac12}^{p_0}h_{1+\frac12}^{p_1}\cdots h_{s_1+\frac12}^{p_{s_1}} w_{\varphi_m})    \\
=&\varphi_m({\bf l})^t\varphi_m(h_{\frac12})^{p_0}\varphi_m(h_{1+\frac12})^{p_1}\cdots \varphi_m(h_{s_1+\frac12})^{p_{s_1}}
\pi(h_{-r_1-\frac12}^{q_{r_1}}\cdots h_{-1-\frac12}^{q_1}h_{-\frac12}^{q_0} w_{\varphi_m})\\
=&\varphi_m({\bf l})^t\varphi_m(h_{\frac12})^{p_0}\varphi_m(h_{1+\frac12})^{p_1}\cdots \varphi_m(h_{s_1+\frac12})^{p_{s_1}}h_{-r_1-\frac12}^{q_{r_1}}\cdots h_{-1-\frac12}^{q_1}h_{-\frac12}^{q_0} w_{\varphi_m}\otimes w_{\varphi'_m}\\
=&h_{-r_1-\frac12}^{q_{r_1}}\cdots h_{-1-\frac12}^{q_1}h_{-\frac12}^{q_0} \varphi_m({\bf l})^t\varphi_m(h_{\frac12})^{p_0}\varphi_m(h_{1+\frac12})^{p_1}\cdots \varphi_m(h_{s_1+\frac12})^{p_{s_1}}w_{\varphi_m}\otimes w_{\varphi'_m}\\
=&h_{-r_1-\frac12}^{q_{r_1}}\cdots h_{-1-\frac12}^{q_1}h_{-\frac12}^{q_0} {\bf l}^th_{\frac12}^{p_0}h_{1+\frac12}^{p_1}\cdots h_{s_1+\frac12}^{p_{s_1}}w_{\varphi_m}\otimes w_{\varphi'_m}\\
=&hw_{\varphi_m}\otimes w_{\varphi'_m}.
\end{split}
\end{equation*}
This proves Claim \ref{cla3.2.1}.

Next, we will show that $\pi$ is a $\DD^{(m,-\infty)}$-module homomorphism.
Notice that for any $$v=h_{-k-\frac12}^{q_k}\cdots h_{-1-\frac12}^{q_1}h_{-\frac12}^{q_0}w_{\varphi_m}\in \Ind_{\DD^{(m, 0)}}^{\DD^{(m,-\infty)}}\C w_{\varphi_m}$$ where $q_i,k\in\Z_+, 0\leq i\leq k$, we have the following:
\begin{claim}\label{cla3.2.2}
   $\pi(h_rv)=h_r\pi(v) {\rm \ and \ } \pi(d_nv)=d_n\pi(v),$
for any  $r\in\Z+\frac12, n\in\Z$ with $n\geq m.$
\end{claim}
By Claim \ref{cla3.2.1} and definition of $\pi$, we deduce
\begin{align}
\pi(h_rv)&= \pi(h_rh_{-k-\frac12}^{q_k}\cdots h_{-1-\frac12}^{q_1}h_{-\frac12}^{q_0}w_{\varphi_m})
 =h_rh_{-k-\frac12}^{q_k}\cdots h_{-1-\frac12}^{q_1}h_{-\frac12}^{q_0}w_{\varphi_m}\otimes w_{\varphi'_m}\notag\\
 &=h_r(h_{-k-\frac12}^{q_k}\cdots h_{-1-\frac12}^{q_1}h_{-\frac12}^{q_0}w_{\varphi_m}\otimes w_{\varphi'_m})
 =h_r\pi(h_{-k-\frac12}^{q_k}\cdots h_{-1-\frac12}^{q_1}h_{-\frac12}^{q_0}w_{\varphi_m})\notag\\
 &=h_r\pi(v),\label{eq321}\\
\pi(d_nh_{-k-\frac12}^{q_k}\cdots h_{-1-\frac12}^{q_1}h_{-\frac12}^{q_0}w_{\varphi_m})&=\pi([d_n,h_{-k-\frac12}^{q_k}\cdots h_{-1-\frac12}^{q_1}h_{-\frac12}^{q_0}]w_{\varphi_m}+h_{-k-\frac12}^{q_k}\cdots h_{-1-\frac12}^{q_1}h_{-\frac12}^{q_0}d_nw_{\varphi_m})\notag\\
&=[d_n,h_{-k-\frac12}^{q_k}\cdots h_{-1-\frac12}^{q_1}h_{-\frac12}^{q_0}]w_{\varphi_m}\otimes w_{\varphi'_m}\label{eq322}\\
&\quad+
\varphi_m(d_n)h_{-k-\frac12}^{q_k}\cdots h_{-1-\frac12}^{q_1}h_{-\frac12}^{q_0}w_{\varphi_m}\otimes w_{\varphi'_m}.\notag
\end{align}

Using definitions of $\pi,\varphi_m'$ and equalities (\ref{rep1}), (\ref{rep2}) we see
\begin{equation}\label{eq323}
    \begin{split}
d_n&\pi(h_{-k-\frac12}^{q_k}\cdots h_{-1-\frac12}^{q_1}h_{-\frac12}^{q_0}w_{\varphi_m})
=d_n(h_{-k-\frac12}^{q_k}\cdots h_{-1-\frac12}^{q_1}h_{-\frac12}^{q_0}w_{\varphi_m}\otimes w_{\varphi'_m})\\
&\quad=d_nh_{-k-\frac12}^{q_k}\cdots h_{-1-\frac12}^{q_1}h_{-\frac12}^{q_0}w_{\varphi_m}\otimes w_{\varphi'_m}+h_{-k-\frac12}^{q_k}\cdots h_{-1-\frac12}^{q_1}h_{-\frac12}^{q_0}w_{\varphi_m}\otimes d_n w_{\varphi'_m}\\
&\quad=[d_n,h_{-k-\frac12}^{q_k}\cdots h_{-1-\frac12}^{q_1}h_{-\frac12}^{q_0}]w_{\varphi_m}\otimes w_{\varphi'_m}+h_{-k-\frac12}^{q_k}\cdots h_{-1-\frac12}^{q_1}h_{-\frac12}^{q_0}d_nw_{\varphi_m}\otimes w_{\varphi'_m}\\
&\qquad+h_{-k-\frac12}^{q_k}\cdots h_{-1-\frac12}^{q_1}h_{-\frac12}^{q_0}w_{\varphi_m}\otimes \varphi'_m(d_n) w_{\varphi'_m}\\
&\quad=[d_n,h_{-k-\frac12}^{q_k}\cdots h_{-1-\frac12}^{q_1}h_{-\frac12}^{q_0}]w_{\varphi_m}\otimes w_{\varphi'_m}
+\varphi_m(d_n)h_{-k-\frac12}^{q_k}\cdots h_{-1-\frac12}^{q_1}h_{-\frac12}^{q_0}w_{\varphi_m}\otimes w_{\varphi'_m}.\\
    \end{split}
\end{equation}
Combining equalities (\ref{eq321}), (\ref{eq322}) and (\ref{eq323}),
we have  proven Claim \ref{cla3.2.2}.
Furthermore, $$\pi({\bf c}v)={\bf c}\pi(v),\ \pi({\bf l}v)={\bf l}\pi(v),\ \forall v\in \Ind_{\DD^{(m, 0)}}^{\DD^{(m,-\infty)}}\C w_{\varphi_m}$$
is clear.
So  $$\Ind_{\DD^{(m, 0)}}^{\DD^{(m,-\infty)}}\C w_{\varphi_m}\cong H^{\DD}\otimes \C_{\varphi'_m}^{\DD^{(m,-\infty)}}$$ as $\DD^{(m,-\infty)}$-modules.
Then by the property of induced modules and Lemma \ref{main2.7}, we have
 \begin{equation*}
\begin{split}
 W_{\varphi_m}=\Ind_{\DD^{(m, 0)}}^{\DD}\C w_{\varphi_m}&\cong
 \Ind_{\DD^{(m,-\infty)}}^{\DD}(\Ind_{\DD^{(m, 0)}}^{\DD^{(m,-\infty)}}\C w_{\varphi_m})\cong \Ind_{\DD^{(m,-\infty)}}^{\DD}(H^{\DD}\otimes \C_{\varphi'_m}^{\DD^{(m,-\infty)}})\\
&\cong H^{\DD}\otimes \Ind_{\DD^{(m,-\infty)}}^{\DD}(\C_{\varphi'_m}^{\DD^{(m,-\infty)}})\cong H^{\DD}\otimes W_{\varphi'_m}^{\DD}.
\end{split}
\end{equation*}

(2) From Lemma \ref{main3.1}  we know  that $H^{\DD}$ is an irreducible $\DD$-module. Then,  from (1), Lemmas \ref{main2.5}, \ref{main2.8}, we obtain the statement in (2).

(3) follows from (1) and Lemmas \ref{main2.8}, \ref{main3.1}.
\end{proof}

Now we consider the Whittaker  $\DD$-module $W_{\varphi_m}$ with $\varphi_m({\bf l})=0$.
Let $\al=\sum_{i\in\Z}\frac{2a_i}{2i-1}h_{i-\frac12}$, where $a_i\in\C$ and only finitely many of $a_i$ are nonzero.
Then we have the automorphism $\theta_\al=\exp({\rm ad} \al)$ of $\DD$ such that
\begin{eqnarray*}
&&\theta_\al (d_n)=d_n+\sum_{i\in\Z}a_ih_{n+i-\frac12}+\frac12\sum_{i\in\Z}a_ia_{-n-i+1}{\bf l},\\
&&\theta_\al (h_r)=h_r+a_{-r+\frac12}{\bf l},\\
&&\theta_\al ({\bf c})={\bf c},\quad \theta_\al ({\bf l})={\bf l},
\end{eqnarray*}
for $n\in\Z, r\in\Z+\frac12$.

\begin{proposition}\label{prop.2}
Suppose that $m\geq 1$ and $\varphi_m: \DD^{(m, 0)}\rightarrow \C$ is a   Whittaker  function
with $\varphi_m({\bf l})=0$. Then the Whittaker $\DD$-module
$W_{\varphi_m}$ is irreducible if and only if $\varphi_m(h_{m-\frac12})\ne0$.
\end{proposition}
\begin{proof}$(\Leftarrow)$.  Since $\varphi_m(h_{m-\frac12})\ne 0$,
 we may choose $a_0,a_{-1},\cdots,a_{-m}\in\C$ such that
 \[
\begin{pmatrix}
\varphi_m(d_m)\\
\varphi_m(d_{m+1})\\
\vdots\\
\varphi_m(d_{2m})
\end{pmatrix}=\begin{pmatrix}
\varphi_m(h_{m-\frac12})&\varphi_m(h_{m-\frac32})&\cdots&\varphi_m(h_{-\frac12})\\
0&\varphi_m(h_{m-\frac12})&\cdots&\varphi_m(h_\frac12)\\
0&0&\cdots&\cdots\\
0&0&\cdots&\varphi_m(h_{m-\frac12})
\end{pmatrix}
\begin{pmatrix}
a_0\\
a_{-1}\\
\vdots\\
a_{-m}
\end{pmatrix},
\]
where we denote $\varphi_m(h_{-\frac12})=0$.
Then
\begin{equation}\label{eq1}
 0=\varphi_m(d_n)+\sum_{i=-m}^{0}(-a_i)\varphi_m(h_{n+i-\frac12})+\frac12\sum_{i\in\Z}a_ia_{-n-i+1}\varphi_m({\bf l}),~\forall~ n\geq m,
\end{equation}
since $\varphi_m({\bf l})=0$ and $\varphi_m(d_{2m+j})=\varphi_m(h_{m-\frac12+j})=0$, for all $j\in \N$.
Denote $\al=-\sum_{i=-m}^{0}\frac{2a_i}{2i-1}  h_{i-\frac12}$.
Then we have  the Lie algebra automorphism of $\theta_{\al}: \DD\to \DD$.
Let $W_{\varphi_m}^{\theta_{\al}}$ be the new Whittaker module with action
 $x\circ w_{\varphi_m}=\theta_{\al}(x)w_{\varphi_m}$ for any $x\in\DD$,
 which is equivalent to $W_{\varphi_m}$.
 By definition of $\theta_{\al}$ and equality (\ref{eq1}),
 it is easy to see that $d_k\circ w_{\varphi_m}=0$ for $k>m-1$, and the actions of $\HH$ and ${\bf c}$ are unchanged.
 Therefore, without loss of generality, we may assume that
 \begin{equation}\varphi_m(d_k)=0,~\forall~ k>m-1.\end{equation}
 Denote $W_0=\Ind_{\DD^{(m, 0)}}^{\DD^{(0,0)}} \C w_{\varphi_m}$. Then
 $\{d_{m-1}^{i_{m-1}}\cdots d_0^{i_0}w_{\varphi_m}~|~(i_{m-1},\ldots, i_0)\in \Z_+^m\}$ is a basis of $W_0$ by PBW Theorem.
 It is not hard to see that $h_{m-\frac12}$ acts injectively on $W_0$.
\begin{claim}\label{cla.1}
$W_0$ is an irreducible $\DD^{(0,0)}$-module.
\end{claim}
By PBW Theorem, for any $v\in W_0$, we may write $v$ in the form of
$$v=\sum_{{\bf i}\in \Z_+^m}a_{\bf i}d^{\bfi}w_{\varphi_m},$$
where $d^{\bfi}=d_{m-1}^{i_{m-1}}\cdots d_0^{i_0}$, and only finitely many $a_{\bfi}\in\C$ are nonzero.
Denote by $\supp(v)$ the set of all ${\bf i}\in\Z_+^m$ such that $a_{\bf i}\ne 0$.
Let $\succ$ denote the {\bf reverse lexicographical} total order on $\Z_+^m$, that is, for any ${\bfi,\bfj}\in\Z_+^m$
$${\bf j}\succ {\bf i}\Leftrightarrow {\rm\ there\ exists\ } 0\leq k\leq m-1 {\text\ such\ that\ }
(j_s=i_s,\forall~ 0\leq s<k) {\text\ and\ } j_k>i_k.$$
Let $\deg(v)$ be the maximal element in $\supp(v)$ with respect to the reverse lexicographical total order on $\Z_+^m$.
Note that $\deg(0)$ is not defined and whenever we write $\deg(v)$ we mean that $v\ne 0$.
Suppose $v\in W_0\setminus\C w_{\varphi_m}, \deg(v)={\bf i}$ and $p=\min\{s:i_s\neq 0\}$.
Then it is straightforward to check that
$$\deg((h_{m-p-\frac12}-\varphi_m(h_{m-p-\frac12}))v)={\bf i-\epsilon_p},$$
where $\epsilon_p$ denotes the element $(\cdots,0,1,0,\cdots)\in\Z_+^m$ with $1$ being in the $(p+1)$-th position from right.
Thus $W_0$ is an irreducible $\DD^{(0,0)}$-module. Claim \ref{cla.1} is proved.

Therefore, $W_{\varphi_m}\cong \Ind_{\DD^{(0,0)}}^{\DD} W_0$ is an irreducible $\DD$-module
by Lemma \ref{main2.3} (with $q=0, t=m-1$).

$(\Rightarrow)$.
Assume that $\varphi_m(h_{m-\frac12})=0$.
Let $w=h_{-\frac12}w_{\varphi_m}\in W_{\varphi_m}$.
Then it is clear that ${\bf l}w=0,\ {\bf c}w=\varphi_m({\bf c})w$  and
$$d_{m+i}w=\varphi_m(d_{m+i})w{\rm\,\ and\,\ } h_{i+\frac12}w=\varphi_m(h_{i+\frac12})w,\forall~ i\in\Z_+,$$
since $\varphi_m({\bf l})=0$ and $h_{i-\frac12}w_{\varphi_m}=0$ for $i\geq m$.
Let $\<w\>$ denote the submodule generated by $w$ of $W_{\varphi_m}$.
By PBW Theorem, we know that $\<w\>$ has a basis
$$d_s^{p_s}\cdots d_{m-2}^{p_{m-2}}d_{m-1}^{p_{m-1}} h_{-r-\frac12}^{q_{r}}\cdots h_{-1-\frac12}^{q_1}h_{-\frac12}^{q_0}w,$$
where $p_i,q_j,r\in\Z_+,s\in\Z_{<m}$, for $s\leq i\leq m-1,0\leq j\leq r$.
It is easy to see that $w_{\varphi_m}\notin \<w\>$, which implies that
$\<w\>$ is a nonzero proper submodule of $W_{\varphi_m}$. So $W_{\varphi_m}$ is not irreducible if   $\varphi_m(h_{m-\frac12})=0$.
\end{proof}

Now we combine Proposition \ref{prop.1} and Proposition \ref{prop.2} into the following main theorem.

\begin{theorem}\label{main3.5}
Let  $m\in \N$ and $\varphi_m: \DD^{(m, 0)}\rightarrow \C$
be a Whittaker function.
\begin{enumerate}[$(1)$]
\item If $\varphi_m({\bf l})\ne 0$, then the Whittaker $\DD$-module $W_{\varphi_m}$
is irreducible if and only if
$$\varphi_m(d_{2m})\ne0 \,\,{\text{or}}\,\, 2\varphi_m({\bf l})\varphi_m(d_{2m-1})-\varphi_m(h_{m-\frac12})^2\ne0.
$$
\item If $\varphi_m({\bf l})=0$, then the Whittaker $\DD$-module $W_{\varphi_m}$
is irreducible if and only if  $\varphi_m(h_{m-\frac12})\ne0$.
 \end{enumerate}
\end{theorem}

\section{$\UU(\C d_0)$-free modules over $\DD$}
In this section, we determine the $\DD$-module structure on  $\UU(\C d_0)$. Also, we give the necessary and sufficient conditions for these modules to be irreducible.

For any $\lambda\in\C^*, \al,\be\in\C$, denote by $\Omega(\lambda,\al,\be)=\C[t]$ the polynomial algebra over $\C$
It is not hard to see that we can give $\Omega(\lambda,\al,\be)$ a $\DD$-module structure via the following actions
\begin{eqnarray*}
&&d_m(f(t))=\lambda^m(t+m\al)f(t+m),\\
&&h_r(f(t))=\be\lambda^rf(t+r),\\
&&{\bf c}(f(t))={\bf l}(f(t))=0,
\end{eqnarray*}
for $m\in\Z, r\in\Z+\frac12$ and $f(t)\in\Omega(\lambda,\al,\be)$.
\begin{remark}
In the above actions, we always fix a $\lambda_0\in\C^*$ such that $\lambda=\lambda_0^2$.
Then we denote $\lambda^r=\lambda_0^{2r}$ for $r\in\Z+\frac12$.
\end{remark}
Moreover, we have the following lemma.
\begin{lemma}\label{main4.1}
\begin{enumerate}[$(1)$]
    \item $\Omega(\lambda,0,0)$ has an irreducible submodule $t\Omega(\lambda,0,0)$.
    \item $\Omega(\lambda,\al,\be)$ is an irreducible $\DD$-module if and only if $\al\ne 0$ or $\be\ne 0$.
\end{enumerate}
\end{lemma}
\begin{proof}
It is clear from the irreducibility of $\VV$-module $\Omega(\lambda,\al)$ and some simple computations.
\end{proof}

Now, we state the main result of this section.
\begin{theorem}\label{main4.2}
Let $M$ be a $\UU(\DD)$-module such that $M$, when considered as a $\UU(\C d_0)$-module,   is  free  of rank 1.
Then $M\cong\Omega(\lambda,\al,\be)$ for some $\lambda\in\C^*,\al,\be\in\C$.
Moreover, $M$ is irreducible if and only if $M\cong\Omega(\lambda,\al,\be)$
for some $\lambda\in\C^*,\al,\be\in\C$ with $(\al,\be)\ne (0,0)$.
\end{theorem}
\begin{proof}
It is clear that $M\cong\Omega(\lambda,\al)=\C[t]$ as $\VV$-modules from Lemma \ref{main2.6}, where $\lambda\in\C^*,\al\in\C$
 and $\C[t]$ is the polynomial algebra.
So, we can assume that
 $${\bf c}f(t)=0,\ \ d_mf(t)=\lambda^m(t+m\al)f(t+m),$$
 for $m\in\Z,f(t)\in\C[t]$.
Now, we consider the action of $h_r$ for $r\in\Z+\frac12$. First, it is not hard to see that
\begin{equation}\label{eq4.1}
h_r(f(t))=h_rf(d_0)(1)=f(t+r)h_r(1),
\end{equation}
for any $r\in\Z+\frac12$ and $f(t)\in\C[t]$.
Next, we consider the two cases.

Case 1. $h_r(1)=0$ for some $r\in\Z+\frac12$.

It is easy to get that $h_rM=0$ by equality (\ref{eq4.1}). Then we have $h_sM=0$ for any $s\in\Z+\frac12$
by the defining relations of $\DD$. Thus $\HH M=0$. It shows that $M\cong\Omega(\lambda,\al,0)$ as $\DD$-modules.

Case 2. $h_r(1)\ne0$ for any $r\in\Z+\frac12$.

First, we show the following claim.
\begin{claim}\label{cla.2}
$h_r(1)\in\C^*$ for any $r\in\Z+\frac12$.
\end{claim}
Assume that there exists $r_0\in\Z+\frac12$ such that $h_{r_0}(1)=\sum_{i=0}^{k_{r_0}}a_it^i$
with ${k_{r_0}}>0$ and $a_{k_{r_0}}\ne 0$.
Choose $s\in\Z+\frac12$ such that $sr_0<0$ and $s+r_0\ne 0$.
Denote
$$X(t)=h_{r_0}(1)=\sum_{i=0}^{k_{r_0}}a_it^i {\rm\ \ and\ \ } Y(t)=h_s(1)=\sum_{j=0}^{k_s}b_jt^j,$$
where $b_{k_s}\ne 0$. Now, we have $[h_s,h_{r_0}](1)=0$. Thus
\begin{equation*}
\begin{split}
0=&h_sh_{r_0}(1)-h_{r_0}h_s(1)=h_sX(t)-h_{r_0}Y(t)\\
=&X(t+s)Y(t)-Y(t+r_0)X(t)\\
=&(\sum_{i=0}^{k_{r_0}}a_i(t+s)^i)(\sum_{j=0}^{k_s}b_jt^j)-(\sum_{j=0}^{k_s}b_j(t+r_0)^j)(\sum_{i=0}^{k_{r_0}}a_it^i)\\
=&a_{k_{r_0}}b_{k_s}(sk_{r_0}-r_0k_s)t^{k_{r_0}+k_s-1}+({\rm lower \ degree\ terms\ in \ } t),
\end{split}
\end{equation*}
which implies a contradiction. Claim \ref{cla.2} is proved.

Now it is clear that ${\bf l}M=0$ from Claim \ref{cla.2}. Denote $h_r(1)=a_r\in\C^*$ for $r\in\Z+\frac12$. Then
\begin{equation*}
\begin{split}
-ra_{m+r}=&-rh_{m+r}(1)=[d_m,h_r](1)\\
=&d_mh_r(1)-h_rd_m(1)=a_rd_m(1)-h_rd_m(1)\\
=&a_r(\lambda^m(t+m\al))-(\lambda^m(t+r+m\al))a_r\\
=&-ra_r\lambda^m,
\end{split}
\end{equation*}
which implies that $a_{m+r}=\lambda^ma_r$ for $m\in\Z,r\in\Z+\frac12$.
So, we have a $\DD$-module isomorphism
$$M\rightarrow \Omega(\lambda,\al,\lambda^{-\frac12}a_{\frac12}),\ \ t^i\rightarrow t^i.$$
The remaining part follows from Lemma \ref{main4.1}.
In conclusion, we complete the proof.
\end{proof}
\section{Tensor products of Whittaker modules and $\UU(\C d_0)$-free $\DD$-modules}
In this section, we show that if $\Omega(\lambda,\al,\be)$ and $W_{\varphi_m}$ are irreducible $\DD$-modules,
then tensor product module  $\Omega(\lambda,\al,\be)\otimes W_{\varphi_m}$
is also irreducible for $\lambda\in\C^*,\al,\be\in\C,m\in\Z_+$.   Moreover, we prove that
$\Omega(\lambda,\al,\be)\otimes W_{\varphi_m}\cong\Omega(\lambda_1,\al_1,\be_1)\otimes W_{\varphi_{m_1}}$
if and only if $\Omega(\lambda,\al,\be)\cong\Omega(\lambda_1,\al_1,\be_1)$ and $W_{\varphi_m}\cong W_{\varphi_{m_1}}$.
In fact, we get the more general results.

\begin{proposition}\label{prop5.1}
Let $\lambda\in\C^*,\al,\be\in\C$ with $(\al,\be)\ne (0,0)$, and $V$ be an irreducible restricted module over $\DD$.
Then $\Omega(\lambda,\al,\be)\otimes V$ is an irreducible $\DD$-module.
\end{proposition}
\begin{proof}
It is clear that, for any $v\in V$, there exists $N(v)\in\Z_+$ such that $d_mv=h_rv=0$
for $m,r\geq N(v)$ by the definition of restricted module.
Now, suppose $M$ is a nonzero submodule of $\Omega(\lambda,\al,\be)\otimes V$.
We only need to show $$M=\Omega(\lambda,\al,\be)\otimes V.$$

\begin{claim}\label{cla.3}
There exists a $v\in V\setminus\{0\}$ such that $1\otimes v\in M$.
\end{claim}

Take a nonzero element $w=\sum_{i=0}^st^i\otimes v_i\in M$ such that $v_i\in V, v_s\ne 0$ and $s$ is minimal.

Let $N=max\{N(v_i):i=0,1,\cdots,s\}$. Then $d_m(v_i)=h_r(v_i)=0$, for all $m,r\geq N$ and $i=0,1,\cdots,s$.
If $\be\ne 0$, then we get that
$$\be^{-1}\lambda^{-r}h_rw=\sum_{i=0}^s(t+r)^i\otimes v_i\in M,$$
for $r\geq N$. We can write the right hand side in the form
$$\sum_{i=0}^sr^iw_i\in M,\forall r>N,$$
where $w_i\in \Omega(\lambda,\al,\be)\otimes V$ are independent of $r$. In particular, $w_s=1\otimes v_s$.
Taking $r=\frac12+N,\frac12+(N+1),\cdots,\frac12+(N+s)$, we see that the coefficient matrix of $w_i$ is a Vandermonde matrix.
Thus each $w_i\in M$. In particular, $w_s=1\otimes v_s\in M$.

If  $\be =0$, then $\al\ne 0$. Similarly,  using the action of $d_mw, m>N$ we can show that   Claim \ref{cla.3} holds as well.

\begin{claim}\label{cla.4}
$M=\Omega(\lambda,\al,\be)\otimes V$.
\end{claim}
From Claim \ref{cla.3}, we have that $1\otimes v\in M$ for some nonzero $v\in V$.
Note that
\begin{eqnarray}\label{eq5.1}
d_m(t^j\otimes v)=\lambda^m(t+m\al)(t+m)^j\otimes v=\lambda^mt(t+m)^j\otimes v+\lambda^mm\al(t+m)^j\otimes v,
\end{eqnarray}
for any $j\in\Z_+,m\geq N(v)$.
Thus by induction on $j$, we deduce that $t^j\otimes v\in M$ for all $j\in\Z_+$,
i.e., $\Omega(\lambda,\al,\be)\otimes v\in M$.
Define $$W=\{w\in V:\Omega(\lambda,\al,\be)\otimes w\subseteq M\}.$$ It is clear that $W\ne 0$ since $v\in W$. We need to prove that $W$ is $\DD$-submodule of $V$.
For any $w\in W,m\in\Z,r\in\frac12+\Z$, we compute that
\begin{eqnarray*}
&&d_m(\Omega(\lambda,\al,\be)\otimes w)=d_m\Omega(\lambda,\al,\be)\otimes w+\Omega(\lambda,\al,\be)\otimes d_mw\subseteq M,\\
&&h_r(\Omega(\lambda,\al,\be)\otimes w)=h_r\Omega(\lambda,\al,\be)\otimes w+\Omega(\lambda,\al,\be)\otimes h_rw\subseteq M.
\end{eqnarray*}
Thus
$\Omega(\lambda,\al,\be)\otimes d_mw,\Omega(\lambda,\al,\be)\otimes h_rw\subseteq M$, i.e., $d_mw,h_rw\in W$.
These show that $W$ is a nonzero submodule of $V$. Thus $W=V$ since $V$ is irreducible,
which implies $M=\Omega(\lambda,\al,\be)\otimes V$. We complete the proof.
\end{proof}
\begin{proposition}\label{prop5.2}
Let $\lambda,\lambda_1\in\C^*,\al,\al_1,\be,\be_1\in\C$.
Assume that $V$ and $V_1$ are irreducible restricted modules over $\DD$.
Then $\Omega(\lambda,\al,\be)\otimes V$ and $\Omega(\lambda_1,\al_1,\be_1)\otimes V_1$ are isomorphic as $\DD$-modules
if and only if $(\lambda,\al,\be)=(\lambda_1,\al_1,\be_1)$ and $V\cong V_1$ as $\DD$-modules.
\end{proposition}
\begin{proof}
It is clear that ``if part" is trivial. We only need to show ``only if part".
Let $\psi:\Omega(\lambda,\al,\be)\otimes V\to  \Omega(\lambda_1,\al_1,\be_1)\otimes V_1$  be a module isomorphism.
For any nonzero element $1\otimes v\in \Omega(\lambda,\al,\be)\otimes V$, suppose
$$\psi(1\otimes v)=\sum_{i=0}^pt^i\otimes w_i,$$
where $w_i\in V_1$ and $w_p\ne 0$. Let $N_v=max\{N{(v)},N{(w_0)},N{(w_1)},\cdots,N{(w_p)}\}$.
Then $$d_mv=d_m(w_i)=h_rv=h_r(w_i)=0,$$
for $m,r\geq N_v,i=0,1,\cdots,p$.
\begin{claim}\label{cla.5}
$\lambda=\lambda_1,\al=\al_1,\be=\be_1$.
\end{claim}
Take $m,m_1\geq N_v$, we have
$$(\lambda^{-m}d_m-\lambda^{-m_1}d_{m_1})(1\otimes v)=\al(m-m_1)(1\otimes v).$$
By $(\lambda^{-m}d_m-\lambda^{-m_1}d_{m_1})\psi(1\otimes v)=\al(m-m_1)\psi(1\otimes v)$,
we deduce that
\begin{eqnarray}\label{eq5.2}
\al(m-m_1)\sum_{i=0}^pt^i\otimes w_i=\sum_{i=0}^p\{(\frac{\lambda_1}{\lambda})^{m}(t+m\al_1)(t+m)^i
-(\frac{\lambda_1}{\lambda})^{m_1}(t+m_1\al_1)(t+m_1)^i\}\otimes w_i.
\end{eqnarray}
Thus $((\frac{\lambda_1}{\lambda})^{m}-(\frac{\lambda_1}{\lambda})^{m_1})t^{p+1}\otimes w_p=0$ for $m,m_1\geq N_v$.
So, $\lambda=\lambda_1$. Therefore, the equation (\ref{eq5.2}) becomes
$$
\al(m-m_1)\sum_{i=0}^pt^i\otimes w_i=\sum_{i=0}^pt\{(t+m)^i-(t+m_1)^i\}\otimes w_i
+\sum_{i=0}^p\al_1\{m(t+m)^i-m_1(t+m_1)^i\}\otimes w_i,
$$
for $m,m_1\geq N_v$. Then the coefficient of $t^p\otimes w_p$ in the right side is $(\al_1+p)(m-m_1)$
which implies $\al=p+\al_1$, i.e., $\al-\al_1=p\ge 0$.
Similar, using $\psi^{-1}$ we can get $\al-\al_1\ge 0$.
So $\al=\al_1$ and $p=0$. Thus  $$\psi(1\otimes v)=1\otimes w_0.$$
Note that for any $r\geq N_v$, we have
\begin{eqnarray*}
\be\lambda^r(1\otimes w_0)=\be\lambda^r\psi(1\otimes v)=\psi(h_r(1\otimes v))
=h_r(1\otimes w_0)
=\be_1\lambda_1^r(1\otimes w_0).
\end{eqnarray*}
Thus $\be=\be_1$.
Hence, Claim \ref{cla.5} is proved.

Now we define the linear map
$$\xi:V\rightarrow V_1,$$
such that $\psi(1\otimes v)=1\otimes \xi(v)$
for $v\in V$. It is clear that $\xi$ is injective.
Note that for any $r\in\Z+\frac12$,
\begin{equation*}
\begin{split}
\be\lambda^r(1\otimes \xi(v))+1\otimes \xi(h_r(v))&=\psi(h_r(1\otimes v))=h_r(\psi(1\otimes v))\\
&=h_r(1\otimes \xi(v))=\be\lambda^r(1\otimes \xi(v))+1\otimes h_r(\xi(v)),
\end{split}
\end{equation*}
which implies
\begin{eqnarray}\label{eq5.3}
\xi(h_r(v))=h_r(\xi(v)), \forall v\in V,r\in\Z+\frac12.
\end{eqnarray}
Next, we consider the equation
$\psi(d_m(1\otimes v))=d_m(\psi(1\otimes v))$
for $m\geq N_v$. We deduce
\begin{eqnarray*}
\lambda^m\psi(t\otimes v)+m\al\lambda^m(1\otimes\xi(v))=\lambda^mt\otimes \xi(v)+m\al\lambda^m(1\otimes\xi(v)).
\end{eqnarray*}
So, $\psi(t\otimes v)=t\otimes \xi(v)$.
Hence $$\psi(d_n(1)\otimes v)=d_n(1)\otimes \xi(v),$$ for $n\in\Z$.
Using the equation $\psi(d_n(1\otimes v))=d_n\psi(1\otimes v)$ for $n\in\Z,v\in V$,
we deduce $$1\otimes \xi(d_n(v))=1\otimes d_n(\xi(v)).$$
That shows
\begin{eqnarray}\label{eq5.4}
\xi(d_n(v))=d_n(\xi(v)), \forall v\in V, n\in\Z.
\end{eqnarray}
It is clear that
$$
\psi({\bf c}(1\otimes v))={\bf c}\psi(1\otimes v){\rm \ and \ }
\psi({\bf l}(1\otimes v))={\bf l}\psi(1\otimes v),
$$
imply that
\begin{eqnarray}\label{eq5.5}
\xi({\bf c}(v))={\bf c}\xi(v) {\rm\ and\ } \xi({\bf l}(v))={\bf l}\xi(v).
\end{eqnarray}
Therefore, $\xi$ is a $\DD$-module  homomorphism by equations (\ref{eq5.3}), (\ref{eq5.4}), and (\ref{eq5.5}).
Since $\xi(V)\ne 0$ and $V_1$ is irreducible, so $\xi(V)=V_1$.
Thus $\xi$ is a $\DD$-module  isomorphism from $V$ to $V_1$.
In conclusion, we complete the proof.
\end{proof}

It is clear that Whittaker $\DD$-modules $W_{\varphi_m}$ are restricted modules over $\DD$.
So we have the following two main theorems from Theorems \ref{main3.5}, \ref{main4.2}
and Propositions \ref{prop5.1}, \ref{prop5.2}.
\begin{theorem}\label{main5.3}
Let $m\in \N,\lambda\in\C^*,\al,\be\in\C$, and $\varphi_m:\DD^{(m, 0)}\rightarrow \C$ be a Whittaker function.
\begin{enumerate}[$(1)$]
    \item If $\varphi_m({\bf l})\ne 0$, then $\Omega(\lambda,\al,\be)\otimes W_{\varphi_m}$ is an irreducible $\DD$-module
    if and only if $$(\varphi_m(d_{2m}), 2\varphi_m({\bf l}) \varphi_m(d_{2m-1})-\varphi_m(h_{m-\frac12})^2)\ne(0,0)
    {\rm\ and\ } (\al,\be)\ne(0,0).$$
    \item  If $\varphi_m({\bf l})=0$, then $\Omega(\lambda,\al,\be)\otimes W_{\varphi_m}$ is an irreducible $\DD$-module
    if and only if
    $$\varphi_m(h_{m-\frac12})\ne0 {\rm\ and\ }(\al,\be)\ne(0,0).$$
\end{enumerate}
\end{theorem}
\begin{theorem}\label{main5.4}
Let $m,m_1\in \Z_+,\lambda,\lambda_1\in\C^*,\al,\al_1,\be,\be_1\in\C.$
Suppose that $\varphi_m:\DD^{(m, 0)}\rightarrow \C$ and $\varphi_{m_1}:\DD^{(m_1, 0)}\rightarrow \C$ be the Whittaker functions such that the corresponding
Whittaker $\DD$-modules $W_{\varphi_{m_1}}$ and $W_{\varphi_{m_1}}$ are all irreducible.
The following results hold.
\begin{enumerate}[$(1)$]
    \item $\Omega(\lambda,\al,\be)\cong\Omega(\lambda_1,\al_1,\be_1)$ if and only if
    $(\lambda,\al,\be)=(\lambda_1,\al_1,\be_1)$.
    \item $\Omega(\lambda,\al,\be)\otimes W_{\varphi_m}\cong\Omega(\lambda_1,\al_1,\be_1)\otimes W_{\varphi_{m_1}}$
    if and only if $(\lambda,\al,\be)=(\lambda_1,\al_1,\be_1)$ and $W_{\varphi_m}\cong W_{\varphi_{m_1}}$.
\end{enumerate}
\end{theorem}
\begin{remark}
Using Theorems \ref{main3.5}, \ref{main4.2} and \ref{main5.3} we can obtain many new irreducible non-weight modules over $\DD$.
\end{remark}

{\bf Acknowledgments:}
The research in this paper was carried out during the visit of the first
author at Wilfrid Laurier University in 2019-2021.
The first author is partially supported by CSC of China (No. 201906340096), and NSF of China (Grants 11771410 and 11931009).
The second author is partially supported by NSF of China (Grant 11801066).
The third author is partially supported by  NSF of China (Grant 11871190) and NSERC (311907-2020).
The authors would like to thank the referees for valuable suggestions to improve the paper.

\end{document}